\documentclass{amsart}
\usepackage{amsmath, amssymb}
\usepackage{xcolor}
\newtheorem{Theorem}{Theorem}[section]

\newtheorem{Lemma}[Theorem]{Lemma}
\newtheorem{Proposition}[Theorem]{Proposition}
\theoremstyle{definition}
\newtheorem{Definition}[Theorem]{Definition}
\theoremstyle{remark}
\newtheorem{rem}[Theorem]{Remark}

\numberwithin{equation}{section}
\newcommand{\norm}[1]{\left\Vert\right\Vert}
\newcommand{\abs}[1]{\left\vert\right\vert}
\newcommand{\set}[1]{\left\{\right\}}
\newcommand{\R}{\mathbb R}
\newcommand{\N}{\mathbb N}
\newcommand{\Z}{\mathbb Z}

\newcommand{\C}{\mathbb C}
\newcommand{\F}{\mathcal{F}}

\newcommand{\res}{\hbox{\rm res}}
\newcommand{\tr}{\hbox{\rm tr}}

\begin{document}

\title[Closed cocycles on PDOs]{On a class of closed cocycles for algebras of non-formal, possibly unbounded, pseudodifferential operators}%
\author{Jean-Pierre Magnot}%
\address{LAREMA - UMR CNRS 6093 \\ Universit\'e d'Angers \\ 2 Boulevard Lavoisier 
	49045 Angers cedex 01 \\ and \\ Lyc\'ee Jeanne dArc \\ 30 avenue de Grande Bretagne\\
	F-63000 Clermont-Ferrand \\
	 http://orcid.org/0000-0002-3959-3443}%
\email{jean-pierr.magnot@ac-clermont.fr}%



\begin{abstract}
In this article, we consider algebras $\mathcal{A}$ of non-formal pseudodifferential operators over $S^1$ which contain $C^\infty(S^1),$ understood as multiplication operators. We apply a construction of Chern-Weil type forms in order to get $2k-$closed cocycles. For $k=1,$ we obtain a cocycle on the algebra of (maybe non classical) pseudodifferential operators with the same cohomology class as the Schwinger cocycle on the algebra of Classical pseudodifferential operators, previously extended and studied by the author on algebras of the same type. 
\end{abstract}

\maketitle
\noindent
{\small MSC (2020) : 17B56, 22E66, 47G30, 58B25, 58J40}

\noindent
{\small Keywords : pseudodifferential operators, cocycle, renormalized traces}

\section*{Introduction}
We present here a construction of a family of cocycles on the Lie algebra of maybe unbounded, maybe non classical, non formal pseudodifferential opeartors $PDO(S^1,V).$ Cocycles on algebras of pseudodifferential operators have been studied from the viewpoint of algebras of formal symbols, see e.g. \cite{KW,KK} for cocycles related to our study, or algebras of non-formal but classical pseudo-differential operators, see e.g. \cite{Rad}. In the last study, the notion of renormalized trace plays an important role, as well as in e.g. the works \cite{CDMP,CDP,Mick} where these cocycles are shown to be linked with anomalies in physics via  differential geometric considerations. 

In our works \cite{Ma2003,Ma2006-2,Ma2008}, we made more precise the link between various aspects: 
\begin{itemize}
	\item the Kravchenko-Khesin cocycle \cite{KK} on formal pseudodifferential operators over $S^1$
	\item the index cocycle on the restricted linear group $GL_{res}$ defined in \cite{PS}
	\item the approach by Radul \cite{Rad}
	\item the Schwinger cocycle \cite{Sch}, see e.g. \cite{CDMP,Mick},
\end{itemize}   
and we showed \cite{Ma2008} that the Schwinger cocycle and the index cocycle could be extended to the algebra $PDO(S^1,V)$ of maybe unbounded, maybe on classical, non formal pseudodifferential operators over $S^1.$ 

We come back to this program in the present work by adding a new idea to our investigations: integrate the classical formulas for Chern-Weil forms $\tr \Omega^k$ in order to define (closed) $2k-$cocycles on $PDO(S^1,V).$ For this task, the classical Bianchi identity is obtained algebraically on the Lie algebra $PDO(S^1,V)$ and the connection that we consider is with values in smoothing operators. 
This enables us to consider $\tr$ as the classical trace of trace class operators, even if technical steps of our investigations require zeta-renormalized traces along the lines of \cite{PayBook,Scott}. In this framework, classical computations of Chern-Weil forms apply and we show that $\tr \Omega^k$ is a $2k-$cocycle on $PDO(S^1,V)$ for $k \in \N^*.$ 

For $k=1,$ even if our connection gives a slightly different $2-$cocycle $\tr \Omega$ from the one studied in \cite{Ma2008}, we show that this $2-$cocycle related to the cohomology class of the Schwinger cocycle. We have to remark that we are here able to prove that these two cocycles have the same cohomology class (up to a factor) on the algebra of \textit{classical} pseudodifferential operators, but we have no information of this kind for $PDO(S^1,V)$ (except that  $\tr \Omega$  is not a coboundary), except that it is not a coboundary.


\section{Preliminaries}
{{} There are multiple frameworks that call, under the name pseudodifferential operator, various objects which may not be (true) operators acting on sections of a finite dimensional complex vector bundle. These last ones are called \textit{formal}, while the others are called \textit{non formal.} Moreover, there exists restricted classes of pseudo-differential operators: \textit{classical,} \textit{log-polyhomogeneous} among others. For the sake of clarity and for a comprehensive exposition, we start by a non-technical presentation of the operators that we consider, and a more rigorous description  will follow. }

\subsection{A panoramic overview on pseudodifferential operators}

We specialize below  to the trivial complex vector bundle $S^1 \times V$ in which $V$ is a $d-$dimensional complex vector space. 
The following definition appears in \cite[Section 2.1]{BGV}.

\begin{Definition} \label{def:diff-op}
	The graded algebra of differential operators acting on the space of smooth sections $C^\infty(S^1,V)$ is the 
	algebra $DO(S^1,V)$ generated by:
	
	$\bullet$ Elements of $C^\infty(S^1,M_d(\C))$
	
	$\bullet$  {Covariant derivation} operators
	$$\nabla_X : g \in C^\infty(S^1,E) \mapsto \nabla_X g$$ where $\nabla$       		is a {smooth } connection on $E$ and $X$ is a {smooth} vector field on $S^1$.
\end{Definition}

{We assign the order $0$  to smooth function} multiplication operators. 
{The derivation} operators have the  
order 1.
We denote by $DO^k(S^1,V)$,$k \geq 0$, the differential operators of order less or equal than $k$.
The algebra $DO(S^1,V)$ is {filtered  by the} order. It is a subalgebra of the algebra of classical pseudo-differential operators $Cl(S^1,V)$ that we describe shortly hereafter, focusing on its necessary aspects.
This is an algebra that contains, for example, the square root of the Laplacian \begin{equation} \label{eq:integral}|D| = \Delta^{1/2} = \int_{\Gamma} \lambda^{1/2}(\Delta-\lambda Id)^{-1} d\lambda,\end{equation}
where $\Delta = - \frac{d^2}{dx^2}$ is the positive Laplacian and $\Gamma$ is a contour around the spectrum of the Laplacian, see e.g. \cite{See,PayBook} for an exposition on contour integrals of pseudodifferential operators. $Cl(S^1,V)$ contains also the inverse of $Id+\Delta,$ and all
smoothing operators on $L^2(S^1,V). $ 

 {{} Wider classes of pseudo-differential operators can be also considered. An example of frequent use remains on the real powers of the Laplacian $\Delta^{\alpha},$ where $\alpha \in \R,$ defined through contour integrals like $\Delta^{1/2}$ already mentioned. One can define the same way $\log \Delta.$ } Pseudodifferential operators (maybe non-classical) are linear operators acting on $C^\infty(S^1,V)$ which reads locally as
$$ A(f) = \int e^{ix.\xi}\sigma(x,\xi) \hat{f}(\xi) d\xi$$ where $\sigma \in C^\infty(T^*S^1, M_n(\C))$ satisfying additional estimates on its partial derivatives {{} which will be given in a detailed form next section} {and $\hat{f}$ means the Fourier transform of $f$}.

{{} Another special class of pseudo-differential operators is also of great interest. This is the set of smoothnig pseudo-differential operators. They are equivalently: 
	\begin{itemize}
		\item classical pseudo-differential operators with order $-\infty,$ that is, they are in the set of classical pseudo-differential operators $k$ for any $k \in \Z.$
		\item pseudodifferential operators defined on $L^2(S^1,V)$ with values in $C^\infty(S^1,V),$ which explains the terminology.
		\item operators with a smooth kernel $K \in C^\infty(S^1 \times S^1, M_n(\C)).$ 
	\end{itemize}
	These operators from an ideal in any algebra of pseudo-differential operators. 
	Quotienting by this ideal, we obtain  algebras of \emph{formal} pseudo-differential operators.}

\subsection{Rigorous approach to pseudodifferential operators}

 Basic facts on pseudo-differential operators defined 
on a vector bundle $E \rightarrow S^1$ can be found e.g. in \cite{Gil}.
{{}
	Let us now recall them in a rigorous exposition, as brief as possible, concentrating our exposition on pseudo-differential operators over $S^1$ which is equipped with an atlas in which changes of coordinates are affine maps. This is made possible through the structure of abelian group of $S^1,$ where the atlas considered is the atlas obtained through the exponential map. 
	We set $S^1 = \left\{z \in {\mathbb{C}} \, | \, |z| = 1\right\}$. We shall use for convenience the smooth atlas 
	$\mathcal{A}$ of $S^1$ defined as follows:
	\begin{eqnarray*}
		\mathcal{A} & = & \{\varphi_0,\varphi_1\} ; \\
		\varphi_n & : & x \in ]0 ; 2\pi[ \mapsto e^{i(x + n\pi)} \subset S^1 
		\hbox{ for } n\in \{0;1\} \end{eqnarray*}     
	Associated to this atlas, we fix a smooth partition of the unit $\{s_0;s_1\}$.
	We identify each of these functions with its associated multiplication operator when 
	necessary.
	An operator $A : C^\infty(S^1,\mathbb{C}) \rightarrow C^\infty(S^1,\mathbb{C})$ can be described in 
	terms of 4 operators
	$$ A_{m,n} : f \mapsto s_m \circ A \circ s_n \hbox{ for } (m,n) \in \{0,1\}^2.$$ 
	A scalar
	pseudo-differential operator  of order $o$ 
	is an operator $$ A : C^\infty(S^1, \mathbb{C}) \rightarrow C^\infty(S^1,\mathbb{C})$$
	such that, $\forall (m,n) \in \{0,1\}^2,$
	$$ A_{m,n}(f) = \int_{]0 ; 2\pi[} e^{-ix\xi}\sigma_{m,n}(x,\xi) \hat{(s_n.  
		f)} (\xi) d\xi$$
	where $\sigma_{m,n} \in C^\infty( ]0 ; 2\pi[ \times \R, \mathbb{C})$ satisfies
	$$\forall (\alpha, \beta) \in \N^2, \quad |D^\alpha_x D^\beta_\xi \sigma_{m,n}(x,\xi)
	|\leq C_{\alpha,\beta}(1 + |\xi|)^{o-\beta}.$$
}
 We note by $PDO(S^1,V)$ the space of maybe non classical, maybe unbounded, pseudodifferential operators acting on $C^\infty(S^1,V).$A pseudo-differential operator of order $o$ is called 
\textbf{classical} if and only if its symbols 
$\sigma$ have an asymptotic expansion
$ \sigma(x,\xi) \sim_{|\xi| \rightarrow +\infty} \sum_{j=-\infty}^o \sigma_j(x,\xi),$
where the maps $\sigma_j : S^1 \times \R^* \rightarrow \mathbb{C}$, 
called \textbf{partial symbols},
are j-positively 
homogeneous, i.e. {{} $$\forall t>0, (x,\xi) \in S^1 \times \R^*, 
(\sigma)_j(x,t\xi) = t^j (\sigma)_j(x,\xi).$$

The class of log-polyhomogeneous pseudo-differential operators can be also described this way. These are operators in $Cl(S^1,V)[\log(\Delta)],$ which inherits also a second degree,  from the evaluation of $Cl(S^1,V)-$polynomials at $\log(\Delta).$}

{Pseudodifferential operators can be also described by their kernel $$K(x,y) = \int_{\R} e^{i(x-y)\xi} \sigma(x,\xi)d\xi$$ which is off-diagonal smooth.} Pseudodifferential operators {with infinitely smooth kernel (or "smoothing" operators).} 

 The quotient 
$$\F Cl(S^1,V)=Cl(S^1,V)/PDO^{-\infty}(S^1,V)$$ of the algebra of pseudo-differential operators by $PDO^{-\infty}(S^1,V)$ forms the algebra of formal {{} classical} pseudo-differential operators. 

\begin{rem}
	Through identification of $\F Cl(S^1,V)$ with the corresponding space of formal symbols, the space $\F Cl(S^1,V)$ is equipped with the natural locally convex topology inherited from the space of formal symbols. 
	A formal symbol $\sigma_k$ is a smooth function in  $C^\infty(T^*S^1\setminus S^1, M_n(\C))$ which is $k-$homogeneous (for $k>0)$), and hence with an element of $C^\infty(S^1, M_n(\C))^2$ evaluating 
	$\sigma_k$ at $\xi = 1$ and $\xi = -1,$ see \cite{Ma2006-2,MRu2021-1}. Identifyting $\F Cl^d(S^1,V)$ with 
	$$ \prod_{k \leq d} C^\infty(S^1, M_n(\C))^2, $$ 
	the vector space $\F Cl^d(S^1,V)$ is a Fr\'echet space, 
	and hence $$\F Cl(S^1,V) = \cup_{d \in \Z}\F Cl^d(S^1,V)$$ 
	is a locally convex topological algebra. 
	

  \end{rem}
	\subsection{The splitting with induced by the connected components of $T^*S^1\setminus S^1.$} \label{ss:+-}
In this section, we define two ideals of the algebra $\mathcal{F}Cl(S^1,V)$, 
that we call $\mathcal{F}Cl_+(S^1,V)$ and $\mathcal{F}Cl_-(S^1,V)$, such that $\mathcal{F}Cl(S^1,V) = \mathcal{F}Cl_+(S^1,V) \oplus \mathcal{F}Cl_-(S^1,V)$. 
This decomposition is explicit in \cite[section 4.4., p. 216]{Ka}, and we give an explicit description here following \cite{Ma2003,Ma2006-2}. 

\begin{Definition}
	
	Let $\sigma$ be a partial symbol of order $o$ on $E$. Then, we define, for $(x,\xi) \in T^*S^1\setminus S^1$, 
	$$ \sigma_+(x,\xi) = \left\{ 
	\begin{array}{ll}
		\sigma(x,\xi) & \hbox{ if $ \xi > 0$} \\
		0 & \hbox{ if $ \xi < 0$} \\
	\end{array}
	\right. \hbox{ and }
	\sigma_-(x,\xi) = \left\{ 
	\begin{array}{ll}
		0 & \hbox{ if $ \xi > 0$} \\
		\sigma(x,\xi) & \hbox{ if $ \xi < 0$} . \\
	\end{array}
	\right.$$
	We define  $p_+(\sigma) = \sigma_+$ and $p_-(\sigma) = \sigma_-$ . 
\end{Definition}
The maps 
$ p_+ : \mathcal{F}Cl(S^1,V) \rightarrow \mathcal{F}Cl(S^1,V) $ { and } $p_- : \mathcal{F}Cl(S^1,V) \rightarrow \mathcal{F}Cl(S^1,V)$ are clearly smooth algebra morphisms (yet non-unital morphisms) that leave the order invariant and are also projections (since multiplication on formal symbols is expressed in terms of pointwise multiplication of tensors). 

\begin{Definition} We define
	$  \mathcal{F}Cl_+(S^1,V) = Im(p_+) = Ker(p_-)$
	and $  \mathcal{F}Cl_-(S^1,V) = Im(p_-) = Ker(p_+).$ \end{Definition}
Since $p_+$ is a projection,  we have the splitting
$$ \mathcal{F}Cl(S^1,V) = \mathcal{F}Cl_+(S^1,V) \oplus \mathcal{F}Cl_-(S^1,V) .$$
Let us give another characterization of $p_+$ and $p_-$. {The operator $D = {-i} \frac{d}{dx}$ splits $C^\infty(S^1, \C^n)$ into three spaces :
	
	- its kernel $E_0,$ built of constant maps
	
	- $E_+$, the vector space spanned by eigenvectors related to positive eigenvalues
	
	- $E_-$, the vector space spanned by eigenvectors related to negative eigenvalues.
	
	The $L^2-$orthogonal projection on $E_0$ is a smoothing operator, which has null formal symbol. By the way, concentrating our attention on thr formal symbol of operators, we can ignore this projection and hence we work on $E_+ \oplus E_-$. When dealing with non-formal operators, we shall set $p_+ = p_{E_+}.$ The following elementary result will be useful for the sequel.  
	\begin{Lemma} \label{l1} \cite{Ma2003,Ma2006-2}
		
		
		
		
		Let $p_{E_+}$ (resp. $p_{E_-}$) be the projection on $E_+$ (resp. $E_-$), then 
		$\sigma(p_{E_+}) ={1 \over 2}(Id + {\xi \over |\xi|})$ and $\sigma(p_{E_-}) = {1 \over 2}(Id - {\xi \over |\xi|})$.
	\end{Lemma}
	


From this, we have the following result.

\begin{Proposition} \label{pag} \cite{Ma2003,Ma2006-2}
	Let $A \in \mathcal{F}Cl(S^1,V).$ 
	$ p_+(A) =  \sigma( p_{E_+}) \circ A = A \circ \sigma( p_{E_+})$ and 
	$  p_-(A) =  \sigma( p_{E_-}) \circ A = A \circ \sigma( p_{E_-}).$
\end{Proposition}

\noindent
{\textbf{Notation.} For shorter notations, we note by $A_\pm = p_\pm(A)$ the formal operators defined from another viewpoint by $$\sigma (A_+)(x,\xi) \quad (\hbox{ resp. }\sigma (A_-)(x,\xi) ) = \left\{ \begin{array}{ll}
		\sigma (A)(x,\xi) & \hbox{if }\xi > 0 \quad(\hbox{ resp. }\xi < 0 ) \\
		0 & \hbox{if }\xi < 0  \quad(\hbox{ resp. }\xi > 0 )\\
	\end{array} \right.$$
	We now turn to maybe non classical pseudo-differentail operators, along the lines of \cite{Ma2008}:
	\begin{Proposition} \label{bracksmooth}
		For any $A \in PDO(S^1,V)$, $$[A,p_{E_+}] \in PDO^{-\infty}(S^1,V).$$
	\end{Proposition}
	\subsection{Renormalized traces of classical pseudodifferential operator}
	$S^1\times V$ is equiped this an Hermitian products $<.,.>$,
	which induces the following $L^2$-inner product on $C^\infty(S^1,V):$ 
	$$ \forall u,v \in C^\infty(S^1,V), \quad (u,v)_{L^2} = \int_{S^1} <u(x),v(x)> dx, $$
	where $dx$ is the Riemannian volume. 
	\begin{Definition}
		$Q$ is a \textbf{weight} of order $s>0$ on $E$ if and only if $Q$ is a classical, elliptic,
		self-adjoint, positive pseudo-differential operator acting on 
		smooth sections of $E$.
	\end{Definition}
	Recall that, under these assumptions, the weight $Q$ has a real discrete spectrum, and that 
	all its eigenspaces are finite dimensional. 
	For such a weight $Q$ of order $q$, one can define the complex 
	powers of $Q$ \cite{See}, 
	see e.g. \cite{CDMP} for a fast overview of technicalities. 
	The powers $Q^{-s}$ of the weight $Q$ 
	are defined for $Re(s) > 0$ using with a contour integral,
	$$ Q^{-s} = \int_\Gamma \lambda^s(Q- \lambda Id)^{-1} d\lambda,$$
	where $\Gamma$ is a contour around the real positive axis.
	Let $A$ be a log-polyhomogeneous pseudo-differential operator.  The map
	$\zeta(A,Q,s) = s\in {\mathbb{C}} \mapsto \hbox{tr} \left( AQ^{-s} \right)\in {\mathbb{C}}$ , 
	defined for $Re(s)$ large, extends
	on ${\mathbb{C}}$ to a meromorphic function \cite{Le}. When $A$ is classical, 
	$\zeta(A,Q,.)$ has a simple pole at $0$  
	with residue ${1 \over q} \res_W A$, where $\res_W$ is the Wodzicki
	residue (\cite{W}, see also \cite{Ka}). Notice that the Wodzicki residue extends the Adler trace \cite{Adl} on formal symbols.
	Following textbooks \cite{PayBook,Scott} for the renormalized trace of 
	classical operators, we define
	
	\begin{Definition} \label{d6} Let $A$ be a log-polyhomogeneous pseudo-differential operator. 
		The finite part of $\zeta(A,Q,s)$ at $s = 0$ is called the renormalized trace
		$\tr^Q A$. If $A$ is a classical pseudo-differential operator,
		$$\tr^Q A = lim_{s \rightarrow 0} (\hbox{tr} (AQ^{-s})
		- {1 \over qs} \res_W (A).$$
	\end{Definition}

	If $A$ is trace class acting on $L^2(S^1,{\mathbb{C}}^k)$, 
	$\hbox{tr}^Q{(A)}=\hbox{tr}{(A)}$.
	The functional $\hbox{tr}^Q$ is of course not a trace.
	In this formula, it appears that the Wodzicki residue $\res_W(A).$ 
	
	\begin{Proposition} \label{p5}
		
		\begin{item}
			
			(i) The Wodzicki residue $\res_W$ is a trace on the algebra of
			classical pseudo-differential operators $Cl(S^1,E)$, i.e. $\forall
			A,B \in Cl(S^1,V), \res_W[A,B]=0.$
			
		\end{item}

		\begin{item}
			
			(ii) (local formula for the Wodzicki residue) Moreover, if $A \in Cl(S^1,V)$,
			$$ \res_W A = {1 \over 2\pi} \int_{S^1} \int_{|\xi|=1} tr \sigma_{-1}(x,\xi) d\xi dx = {1 \over 2\pi} \sum_{\xi = \pm 1} \int_{S^1}  tr \sigma_{-1}(x,\xi) d\xi dx. $$
			In particular, $\res_W$ does not depend on the choice
			of $Q$.
		\end{item}

		

	\end{Proposition}
	
	Since $\tr^Q$ is a linear extension of the classical trace $\tr$ of trace-class operators acting on $L^2(S^,V),$ it has weaker properties. Let us summarize some of them which are of interest for our work following first \cite{CDMP}, completed by \cite{Ma2016} for the third point.
	
	\begin{Proposition} \label{p6}

		\begin{itemize}
			
			\item   Given two (classical) pseudo-differential operators A and B,
			given a weight Q,
			\begin{equation}\label{crochet}  \tr^Q[A,B] = -{1 \over q} \res (A[B,\log Q]). \end{equation}
			

			\item Under the previous notations, if C is a classical elliptic
			injective operator or a diffeomorphism, $tr^{C^{-1}QC}\left( C^{-1}AC
			\right)$ is well-defined and equals $\tr^QA$. 
			
		\end{itemize}
		
	\end{Proposition}

	Since $\tr^Q$ is not tracial, let us give one more property on the renormalized trace of the bracket, from e.g. \cite{Ma2020}.

	\begin{Proposition} \label{brackclass}  $$\forall (A,B) \in PDO^{-\infty}(S^1,V) \times Cl(S^1,V), \quad \tr^Q[A,B]= \tr[A,B]=0.$$
	\end{Proposition}

Moreover, we can push further the property on $PDO(S^1,V):$

\begin{Proposition} \label{ncltrace}  $$\forall (A,B) \in PDO^{-\infty}(S^1,V) \times PDO(S^1,V), \quad  \tr[A,B]=0.$$
\end{Proposition}

\begin{proof} Let $(A,B) \in PDO^{-\infty}(S^1,V) \times PDO(S^1,V).$
	Let $o = ord(B).$ Then $(1+|D|)^{o+2}$ is a pseudodifferential operator of order $o+2$ with inverse $(1+|D|)^{-o-2}.$
	Since $A$ is smoothing, $[A,B]$ is also smoothing and hence trace class. Therefore, the following expression makes sense and we can compute (since we are commuting operators of order -2 at most, which are trace class anyway): 
	\begin{eqnarray*}
		\tr(AB) & = & \tr(A (1+|D|)^{o+2} (1+|D|)^{-o-2} B ) \\ & = & \tr((1+|D|)^{-o-2} B A (1+|D|)^{o+2}  )
		\\ & = & \tr( B A (1+|D|)^{o+2} (1+|D|)^{-o-2} )
		\\ & = & \tr(BA).
	\end{eqnarray*}
\end{proof}
\section{A family of cocycles}
{{}
We consider the Lie algebra cohomology of $PDO(S^1,V)$ and of its Lie subalgebras.	
}
\begin{Definition}
	We define on $PDO(S^1,V)$  $$\theta: (a,b) \in PDO(S^1,V) \mapsto \theta_a (b)= a \circ p_{E_+} \circ b.$$  
	\end{Definition}
{{} In order to understand better our construction, we have to precise that the reader has to understand $\theta$ as a 1-form
	$$ a \mapsto \theta_a = (a \circ p_{E_+}) \circ (.)$$ with values in linear maps on $PDO(S^1,V),$ in order to understand better the link of this construction with the theory of connection 1-forms. This linear map is only a composition operator, in the spirit of the connection 1-forms described in \cite{Ma2020}. 
\begin{rem} On formal symbols, $$\sigma\left(\theta(a)\right) = \sigma(a)_+.$$\end{rem}

Therefore, we consider the curvature $\Omega$ of the connection $\theta,$ that is, 
$$\Omega(a,b) = \theta_a \theta_b - \theta_b \theta_a - \theta_{[a,b]} $$
which is, like $\theta_a,$ is an operator of coomposition on the left, i.e. $$\Omega(a,b)= s(a,b) \circ (.).$$
} 

\begin{Proposition}
	 $\Omega$ is a $PDO^{-\infty}(S^1,V)-$valued 2-form, in the sense that . $$\forall (a,b,c)\in PDO(S^1,V)^3, \quad \Omega(a,b)c = s c$$ where $s(a,b) \in PDO^{-\infty}(S^1,V).$
\end{Proposition}

\begin{proof} Let $(a,b,c)\in PDO(S^1,V)^3.$ From Proposition \ref{bracksmooth}, $$a p_{E_+} b p_{E_+} - b p_{E_+} a p_{E_+} = [a,b] p_{E_+} + s(a,b)$$
	where $$s(a,b) = a [p_{E_+},b]p_{E_+} -  b [p_{E_+},a]p_{E_+}$$ is a smoothing operator. Therefore, 
	\begin{eqnarray*}
		\Omega(a,b)c & = &  (\theta_a \theta_b - \theta_b \theta_a - \theta_{[a,b]})c\\
		& = & (a p_{E_+} b p_{E_+} - b p_{E_+} a p_{E_+}c - [a,b] p_{E_+})c \\
		& = & s(a,b) c
	\end{eqnarray*}
	\end{proof}
{{} We now consider 
$$\Omega^k  (a_1,...,a_{2k}) =  \Omega \wedge ... \wedge \Omega(a_1,...,a_{2k}) = s^k(a_1,...,a_{2k})\circ (.).$$
with
$$s^k(a_1,...,a_{2k}) = \sum_{\sigma \in \mathfrak{G}_{2k}} \frac{(-1)^{\epsilon(\sigma)}}{2k !}  \prod_{i = 1}^k s(a_{\sigma(2i-1)},a_{\sigma(2i)})$$
Since $s$ is with values in smoothing operators, so is the $2k-$form $s^k.$ By slight abuse of notations, we define the $2k-$form with values in $\C:$ 
\begin{equation} \label{trace} tr(\Omega^k) = tr(s^k) \end{equation}
where $tr$ is the trace of trace-class operators on $L^2(S^1,V).$
Since $\Omega(a,b) = s(a,b) \circ (.)$ is the composition on the left by a smoothing operator, we will now understand $tr(\Omega^k),$ and the similar expressions, along the lines of (\ref{trace}).}

This is the main property to get the following theorem

\begin{Theorem}
	The Chern-Weil like forms $$\tr (\Omega^k)$$ define closed $2k-$cocycles in Lie algebra cohomology of $PDO(S^1,V).$   
	\end{Theorem}
\begin{proof}
	We compute directly {{}the coboundary}
	\begin{eqnarray*} d tr \Omega^k & = & tr d\Omega^k \end{eqnarray*} 
	and reduce the computation to 
	$$ d \Omega = -[\theta,\Omega].$$
	Therefore, 
	\begin{eqnarray*} d tr \Omega^k & = & -tr [\theta,\Omega^k] \\
	& = & 0 \end{eqnarray*} applying Proposition \ref{ncltrace}.
	\end{proof}

This result remains valid for any Lie subalgebra $\mathcal{A} \subset PDO(S^1,V).$
Let us now give a key elementary lemma about non-exactness, already applied in \cite{Ma2008}:
\begin{Lemma} \label{nonvanish}
	Let $\mathcal{A}$ be a Lie subalgebra of $PDO(S^1,V).$ Let $c$ be a cocycle on $\mathcal{A}.$ Let $\mathcal{B}$ be a commutative Lie subalgebra of $\mathcal{A},$ i.e. $[\mathcal{B},\mathcal{B}]=\{0\}.$ If $c$ is non vanishing on $\mathcal{B},$ then $c$ is not exact.
\end{Lemma}
\section{$\tr \Omega$ is cohomologous to the Schwinger cocycle on $Cl(S^1,V).$}

\begin{Theorem}
	On $Cl(S^1,V), $ $\tr \Omega$ has the same cohomology class as $\frac{1}{2} c_s,$ where $c_s$ is the Schwinger cocycle. By the way it has non-trivial Lie algebra cohomology class on $PDO(S^1,V).$ 
\end{Theorem}

\begin{proof} First, let $(X,Y) \in Cl(S^1,V)^2.$
	We have that $$\tr \Omega (X,Y) = \tr^\Delta \left[\theta_X,\theta_Y\right] - \tr^\Delta \theta_{[X,Y]}.$$
	The term $\tr^\Delta\theta_{[X,Y]}$ is a coboundary. Let us calculate $\tr^\Delta \left[\theta_X,\theta_Y\right].$ For this, we remark that $\sigma(\theta_X) = \sigma_+(X)$ thus
	$$\tr^\Delta \left[\theta_X,\theta_Y\right] = -\frac{i}{2\pi} res \sigma_+(X)\left[\sigma_+(Y),log \Delta\right].$$
	The last term thus can be identified with the pull-back of the Kravchenko-Khesin-Radul cocycle on $Cl(S^1,V)$ so that it has the same cohomology class as the Schwinger cocycle following \cite{Ma2006-2}.  
	
	If $\tr \Omega$ was a coboundary on $PDO(S^1,V),$ it would be also a coboundary on $Cl(S^1,V).$ So that, $\tr \Omega$ has non-trivial Hochschild cohomology class on $PDO(S^1,V).$ 
	\end{proof}
We have here to remark that the full comparison of $\tr \Omega$ with $\frac{1}{2} c_s$ remains an open question, because the correspondence is only established on $Cl(S^1,V).$

\section{Conclusion}
The family of cocycles on $PDO(S^1,V)$ that we produced show that there can exist some $2k-$cocycles on $PDO(S^1,V).$ Beside the $2-$cocycle $\tr \Omega$ which is cohomologous to the well-known Schwinger cocycle on $Cl(S^1,V)$ we get the following open question: 
\vskip 12pt
\centerline{Prove that there exists other non-trivial cocycles in our family.}
\vskip 12pt

 More generally and algebraically, the full description of the Hochschild cohomoloy of various Lie subalgebras of $PDO(S^1,V)$ (especially those with some unbounded operators) needs to be investigated. From a geometric viewpoint, the meaning of the higher Chern-Weil forms that we describe here, intrinsically liked with the sign of the Dirac operator, carry interpretations that can be only heuristic since the classical differential geometry (with atlases) fail to apply, but the seem intrinsically linked with the integrable almost complex structure described in \cite{MRu2021-1} in the context of formal classical pseudo-differential operators.


\end{document}